\documentclass [11pt,reqno]{amsart}
\usepackage {amsmath,amssymb,verbatim,geometry,color,esint}
\usepackage[all]{xy}
\usepackage{mathrsfs}
\usepackage[pdftex,hyperindex]{hyperref}
\usepackage[pdftex]{graphicx}
\usepackage{tikz}
\usetikzlibrary{math}

\usepackage{accents}

\newcommand{\A}{\mathbb{A}}
\newcommand{\B}{\mathbb{B}}

\newcommand{\Gm}{\mathbb{G}_\mathrm{m}}

\newcommand{\N}{\mathbb{N}}
\renewcommand{\P}{\mathbb{P}}
 \newcommand{\Q}{\mathbb{Q}}
 \newcommand{\R}{\mathbb{R}}
 \newcommand{\Z}{\mathbb{Z}}

\newcommand{\fa}{\mathfrak{a}}
\newcommand{\fb}{\mathfrak{b}}

\newcommand{\cA}{\mathcal{A}}

\newcommand{\cE}{\mathcal{E}}
\newcommand{\cF}{\mathcal{F}}

\newcommand{\cH}{\mathcal{H}}

\newcommand{\cL}{\mathcal{L}}

\newcommand{\cO}{\mathcal{O}}

\newcommand{\cX}{\mathcal{X}}
\newcommand{\cY}{\mathcal{Y}}

\renewcommand{\a}{\alpha}

\newcommand{\e}{\varepsilon}
\newcommand{\f}{\varphi}
\newcommand{\unipar}{\varpi}

\newcommand{\la}{\lambda}
\newcommand{\om}{\omega}

\newcommand{\p}{\psi}

\newcommand{\gf}{\mathrm{gf}}

\newcommand{\an}{\mathrm{an}}

\DeclareMathOperator{\en}{E}

\DeclareMathOperator{\env}{P}
\DeclareMathOperator{\pp}{P}

\DeclareMathOperator{\qq}{Q}

\DeclareMathOperator{\Cz}{C^0}

\DeclareMathOperator{\charac}{char}

\DeclareMathOperator{\Amp}{Amp}
\DeclareMathOperator{\supp}{supp}
\DeclareMathOperator{\vol}{vol}

\DeclareMathOperator{\Pic}{Pic}

\DeclareMathOperator{\ord}{ord}

\DeclareMathOperator{\PSH}{PSH}

\DeclareMathOperator{\PL}{PL}

\DeclareMathOperator{\Hnot}{H^0}

\DeclareMathOperator{\te}{T}

\DeclareMathOperator{\Num}{N^1}

\DeclareMathOperator{\VCar}{VCar}

\newcommand{\supstar}{\mathrm{sup}^\star }

\renewcommand{\div}{\mathrm{div}}

\newcommand{\triv}{\mathrm{triv}}

\newcommand{\lin}{\mathrm{lin}}

\numberwithin{equation}{section}       % Number formulas within sections

\newtheorem{prop} {Proposition} [section]
 
\newtheorem{defi}[prop] {Definition}
\newtheorem{lem}[prop] {Lemma}
\newtheorem{cor}[prop]{Corollary}
\newtheorem{prop-def}[prop]{Proposition-Definition}

\newtheorem*{thmA}{Theorem A} 
 
\newtheorem*{thmB}{Theorem B}

\newtheorem{rmk}[prop]{Remark}

\theoremstyle{remark}

\title[Addendum]{Addendum to the article `Global pluripotential theory over a trivially valued field'}
\date{\today}

\author{S{\'e}bastien Boucksom \and Mattias Jonsson}

\address{CNRS--CMLS\\
  \'Ecole Polytechnique\\
  F-91128 Palaiseau Cedex\\
  France}
\email{sebastien.boucksom@polytechnique.edu}

\address{Dept of Mathematics\\
  University of Michigan\\
  Ann Arbor, MI 48109-1043\\
  USA}
\email{mattiasj@umich.edu}

\begin{document}

\begin{abstract}
This note is an addendum to the paper `Global pluripotential theory over a trivially valued field' by the present authors, in which we prove two results.
  Let $X$ be an irreducible projective variety over an algebraically closed field field $k$, and assume that $k$ has characteristic zero, or that $X$ has dimension at most two.
  We first prove that when $X$ is smooth, the envelope property holds for any numerical class on $X$. Then we prove that for $X$ possibly singular and for an ample numerical class, the Monge--Amp\`ere energy of a bounded function is equal to the energy of its usc regularized plurisubharmonic envelope.\\
  
  Cette note est un appendice au papier `Global pluripotential theory over a trivially valued field' par les pr\'esents auteurs, dans lequel nous prouvons deux r\'esultats. Soit $X$ une vari\'et\'e projective irreductible sur un corps alg\'ebriquement clos $k$, et supposons que $k$ est de caract\'eristique nulle, ou que $X$ est de dimension au plus deux. Nous prouvons d'abord que, lorsque $X$ est lisse, la propri\'et\'e d'enveloppe est valable pour toute classe num\'erique sur $X$. Ensuite, nous prouvons que, pour $X$ possiblement singulier et pour toute classe num\'erique ample, l'\'energie de Monge--Amp\`ere de toute fonction born\'ee est \'egale \`a celle de son enveloppe plurisousharmonique regularisée. 
\end{abstract}

\maketitle

% \setcounter{tocdepth}{1}
% \tableofcontents
% 
%
%%%%%%%%%%%%%%%%%%%%%%%%%%%%%%%%%%%%%%%%%%%%%%%%%%%%%%%%%%%%%%%%%%%
%
%
% 
%
%%%%%%%%%%%%%%%%%%%%%%%%%%%%%%%%%%%%%%%%%%%%%%%%%%%%%%%%%%%%%%%%%%%
%
%
\section*{Introduction}
The purpose of this note is to strengthen two results in the article~\cite{trivval}, where we developed global pluripotential on the Berkovich analytification over a trivially valued field. The results here are used in~\cite{nakstab1new,nakstab2}. One should view the current note as an addendum to~\cite{trivval}, rather than a stand-alone paper.

Let $k$ be an algebraically closed field, and $X$ an irreducible projective variety over $k$. To any numerical class $\theta\in\Num(X)$ we associate a class $\PSH(\theta)$ of $\theta$-psh functions; these are upper semicontinuous functions $\f\colon X^\an\to\R\cup\{-\infty\}$ on the Berkovich analytification of $X$ with respect to the trivial absolute value on $k$.  We say that $\theta$ has the \emph{envelope property} if for any bounded-above family $(\f_\a)_\a$ in $\PSH(\theta)$, the function $\sup^\star_\a\f_\a$ is $\theta$-psh.
\begin{thmA}
  Assume that $X$ is smooth, and that $\charac k=0$ or $\dim X\le 2$. Then any numerical class $\theta\in\Num(X)$ has the envelope property.
\end{thmA}
In~\cite[Theorem~5.20]{trivval}, this was established for nef classes $\theta$ following~\cite{siminag}, and the proof here is not so different.

\medskip
For the second result we allow $X$ to be singular, but work with an \emph{ample} class $\omega\in\Num(X)$. The $\om$-psh envelope $\env_\om(\f)$ of a bounded function $\f\colon X^\an\to\R$ is defined as the supremum of all functions $\p\in\PSH(\om)$ with $\p\le\f$, and the envelope property for $\omega$ is equivalent to \emph{continuity of envelopes} in the sense of $\env_\om(\f)$ being continuous whenever $\f$ is continuous. It is also equivalent to the usc envelope $\env_\om^\star(\f)$ being $\om$-psh for any bounded function $\f$.

In~\cite{trivval} we also defined the \emph{Monge--Amp\`ere energy} $\en_\om(\f)\in\R\cup\{-\infty\}$ of any bounded-above function $\f\colon X^\an\to\R\cup\{-\infty\}$. We did this first for $\om$-psh functions in terms of an energy pairing ultimately deriving from intersection numbers on compactified test configurations, see~\S\ref{sec:enpair} below, then for general bounded-above functions $\f$, setting
$$
\en_\om(\f):=\sup\{\en_\om(\p)\mid \p\in\PSH(\om), \p\le\f\}.
$$
We say that $(X,\om)$ satisfies the \emph{weak envelope property} if there exists a projective birational morphism $\pi\colon\tilde X\to X$ and an ample class $\tilde\om\in\Num(\tilde X)$ such that $(\tilde X,\tilde\om)$ has the envelope property and $\tilde\om\ge\pi^\star\om$ (by which we mean $\tilde\om-\pi^\star\om$ is nef). It follows from~\cite[Theorem~5.20]{trivval} that the weak envelope property holds when $\charac k=0$ or $\dim X\le2$.
\begin{thmB}
  Assume that $\om\in\Num(X)$ is an ample class, and that the weak envelope property holds for $(X,\om)$. Then, for any bounded function $\f\colon X^\an\to\R$, we have 
  \begin{equation*}
    \en_\om(\f)=\en_\om(\env_\om(\f))=\en_\om(\env_\om^\star(\f)).
  \end{equation*}
\end{thmB}
The first equality is definitional, see~\cite[(8.2)]{trivval}, and the second equality follows from~\cite[Proposition~8.3]{trivval} if $\om$ has the envelope property. The main content of Theorem~B is thus the second equality when the envelope property is unknown or even fails (for example, when $X$ is not unibranch).
%
% 
%%%%%%%%%%%%%%%%%%%%%%%%%%%%%%%%%%%%%%%%%%%%%%%%%%%%%%%%%%%%
\subsection*{Acknowledgement}
  The second author was partially supported by NSF grants DMS-1900025 and DMS-2154380.
   
% 
%
%%%%%%%%%%%%%%%%%%%%%%%%%%%%%%%%%%%%%%%%%%%%%%%%%%%%%%%%%%%%%%%%%%%
%
%
\section{Preliminaries}
Throughout the paper, $X$ is an irreducible projective variety over an algebraically closed field $k$.
%
%%%%%%%%%%%%%%%%%%%%%%%%%%%%%%%%%%%%%%%%%%%%%%%%%%%%%%%%%%%%%%%%%%%
%
\subsection{The $\theta$-psh envelope}
Fix any numerical class $\theta\in\Num(X)$. We refer to~\cite[\S4]{trivval} for the definition of the class $\PSH(\theta)$ of $\theta$-psh functions. We have that $\PSH(\theta)$ is nonempty  only if $\theta$ is psef, whereas $\PSH(\theta)$ contains the constant functions iff $\theta$ is nef.
\begin{defi}\label{defi:pshenv} The \emph{$\theta$-psh envelope} of a function $\f\colon X^\an\to\R\cup\{\pm\infty\}$ is the function $\env_\theta(\f)\colon X^\an\to\R\cup\{\pm\infty\}$ defined as the pointwise supremum
$$
\env_\theta(\f):=\sup\left\{\p\in\PSH(\theta)\mid\p\le\f\right\}.
$$
\end{defi}
Thus $\env_\theta(\f)\equiv-\infty$ iff there is no $\p\in\PSH(\theta)$ with $\p\le\f$. When $\theta=c_1(L)$ for a $\Q$-line bundle $L$, we write $\pp_L:=\pp_\theta$.
Despite the name, $\env_\theta(\f)$ is not always $ \theta$-psh (and indeed not even usc in general).  
However, it is clear that 
\begin{itemize}
\item $\f\mapsto\env_ \theta(\f)$ is increasing;
\item $\env_ \theta(\f+c)=\env_ \theta(\f)+c$ for all $c\in\R$.
\end{itemize}
The envelope operator is also continuous along increasing nets of lsc functions: 

\begin{lem}\label{lem:envlsc} If $\f\colon X^\an\to\R\cup\{+\infty\}$ is the pointwise limit of an increasing net $(\f_j)$ of bounded-below, lsc functions, then 
$\env_\theta(\f_j)\nearrow\env_\theta(\f)$ pointwise on $X^\an$. 
\end{lem}
\begin{proof} We trivially have $\lim_j\env_\theta(\f_j)=\sup_j\env_\theta(\f_j)\le\env_\theta(\f)$. Pick $\e>0$ and $\p\in\PSH(\theta)$ such that $\p\le\f$, and hence $\p<\f+\e$. Since $\p$ is usc and the $\f_j$ is lsc, a simple variant of Dini's lemma shows that $\p<\f_j+\e$ for all $j$ large enough, and hence $\p\le\env_\theta(\f_j)+\e$. Taking the supremum over $\p$ yields $\env_\theta(\f)\le\sup_j\env_\theta(\f_j)$, and we are done. 
\end{proof}

As in~\cite[Lemma 7.30]{BE}, the envelope property admits the following useful reformulation. 

\begin{lem}\label{lem:contenv} If $\PSH(\theta)\ne\emptyset$, then the following statements are equivalent:
\begin{itemize}
\item[(i)] $ \theta$ has the envelope property;
\item[(ii)] for any function $\f\colon X^\an\to\R\cup\{\pm\infty\}$, we have 
$$
\env_ \theta(\f)\equiv-\infty,\,\env_ \theta(\f)^\star\equiv+\infty,\text{ or }\env_\theta(\f)^\star\in\PSH(\theta);
$$
\item[(iii)] $\f\in\PL(X)\Longrightarrow\env_ \theta(\f)\in\PSH(\theta)$.
\end{itemize}
\end{lem}
%We refer to the property in~(iv) as \emph{continuity of envelopes}.
\begin{proof} First assume (i). Pick any $\f\colon X^\an\to\R\cup\{\pm\infty\}$, and suppose that the set $\cF:=\left\{\p\in\PSH(\theta)\mid\p\le\f\right\}$ is nonempty, so that $\env_ \theta(\f)\not\equiv-\infty$. If the functions in $\cF$ are uniformly bounded above, then $\env_ \theta(\f)^\star\in\PSH(\theta)$, by (i). If not, choose $\om\in\Amp(X)$ with $\om\ge\theta$, and hence $\cF\subset\PSH(\om)$. By the definition of the Alexander--Taylor capacity, see~\cite[\S4.6]{trivval}, we then have 
$$
\env_ \theta(\f)(v)=\sup\left\{\p(v)\mid \p\in\cF\right\}\ge\sup\left\{\sup\p\mid\p\in\cF\right\}-\te_\om(v)=+\infty
$$ 
for all $v\in X^\div$, and hence $\env_\theta(\f)^\star\equiv+\infty$, by density of $X^\div$. This proves (i)$\Rightarrow$(ii). 

Next we prove (ii)$\Rightarrow$(iii), so pick $\f\in\PL(X)$. Since $\f$ is bounded and $\PSH(\theta)$ is nonempty and invariant under addition of constants, we have $\env_\theta(\f)\not\equiv-\infty$.
%and we evidently also have $\env_\theta(\f)\not\equiv+\infty$. Thus 
Now $\env_\theta(\f)\le\f$ implies $\env_ \theta(\f)^\star\le\f$ since $\f$ is usc. In particular, $\env_\theta(\f)^\star\not\equiv+\infty$, so $\env_\theta(\f)^\star\in\PSH(\theta)$ by~(ii).
Thus $\env_\theta(\f)^\star$ is a competitor in the definition of $\env_ \theta(\f)$, so
$\env_ \theta(\f)=\env_ \theta(\f)^\star$ is $\theta$-psh. 

Finally, we prove (iii)$\Rightarrow$(i), following~\cite[Lemma 7.30]{BE}. Let $(\f_i)$ be a bounded-above family in $\PSH( \theta)$, and set $\f:=\supstar_i\f_i$. Since $\f$ is usc and $X^\an$ is compact, we can find a decreasing net $(\p_j)$ in $\Cz(X)$ such that $\p_j\to\f$. By density of $\PL(X)$ in $\Cz(X)$ wrt uniform convergence (see~\cite[Theorem~2.2]{trivval}), we can in fact assume $\p_j\in\PL(X)$, and hence $\env_\theta(\p_j)\in\PSH(\theta)$, by~(iii). For all $i,j$, we have $\f_i\le\p_j$, and hence $\f_i\le \env_ \theta(\p_j)$, which in turn yields $\f\le \env_ \theta(\p_j)\le\p_j$. We have thus written $\f$ as the limit of the decreasing net of $\theta$-psh functions $\env_\theta(\p_j)$, which shows that $\f$ is $ \theta$-psh.
\end{proof}

\begin{cor}\label{cor:envusc} Assume that $ \theta$ has the envelope property, and consider a usc function $\f\colon X^\an\to\R\cup\{-\infty\}$. Then: 
\begin{itemize}
\item[(i)] $\env_ \theta(\f)$ is $ \theta$-psh, or $\env_ \theta(\f)\equiv-\infty$; 
\item[(ii)] if $\f$ is the limit of a decreasing net $(\f_j)$ of bounded-above, usc functions, then $\env_ \theta(\f_j)\searrow\env_ \theta(\f)$.
\end{itemize}
\end{cor}
\begin{proof} By~Lemma~\ref{lem:contenv}, either $\p:=\env_ \theta(\f)^\star$ is $ \theta$-psh, or $\env_ \theta(\f)\equiv-\infty$ (the latter being automatic if $\PSH(\theta)=\emptyset$). Since $\env_ \theta(\f)\le\f$ and $\f$ is usc, we also have  $\p\le\f$. If $\p$ is $ \theta$-psh, then $\p\le\env_ \theta(\f)$, which proves (i). 

To see (ii), note that $\rho:=\lim_j\env_ \theta(\f_j)$ satisfies either $\rho\in\PSH(\theta)$ or $\rho\equiv-\infty$, by~\cite[Theorem~4.7]{trivval}. Furthermore, $\env_\theta(\f_j)\le\f_j$ yields, in the limit, $\rho\le\f$, and hence $\rho\le\env_\theta(\f)$ (by definition of $\env_\theta(\f)$ if $\rho\in\PSH(\theta)$, and trivially if $\rho\equiv-\infty$). Thus $\lim_j\env_\theta(\f_j)=\rho\le\env_\theta(\f)$. On the other hand, $\env_\theta(\f_j)\ge\env_\theta(\f)$ implies $\rho\ge\env_\theta(\f)$, which completes the proof of (ii). 
\end{proof}
%
%%%%%%%%%%%%%%%%%%%%%%%%%%%%%%%%%%%%%%%%%%%%%%%%%%%%%%%%%%%%%%%%%%%
%
\subsection{The Fubini--Study envelope}
Now consider a big $\Q$-line bundle $L$. Recall~\cite[\S2.4]{trivval} that for any subgroup $\Lambda\subset\R$, $\cH^\gf_\Lambda(L)$ denotes the set of functions $\f\colon X^\an\to\R$ of the form $$\f=m^{-1}\max_j\{\log|s_j|+\lambda_j\},$$ where $m\in\Z_{>0}$ is such that $mL$ is an honest line bundle, $(s_j)_j$ is a finite set of nonzero global sections of $mL$, and $\la_j\in\Lambda$.

We define the \emph{Fubini--Study envelope} of a bounded function $\f\colon X^\an\to\R$ as 
\begin{equation}\label{equ:FSenv}
\qq_L(\f):=\sup\left\{\p\in\cH^\gf_\R(L)\mid\p\le\f\right\}.
\end{equation}
By approximation, $\cH^\gf_\R(L)$ can be replaced by $\cH^\gf_\Q(L)=\cH^\gf_\Z(L)$ in this definition, see~\cite[(2.10)]{trivval}. Note also that $\qq_L(\f)\colon X^\an\to\R\cup\{-\infty\}$ is bounded above and lsc. 
 
Recall that the \emph{augmented base locus} of $L$ can be described as
$$
\B_+(L):=\bigcap\{\supp E\mid E\text{ effective }\Q\text{-Cartier divisor},\,L-E\text{ ample}\},
$$
a strict Zariski closed subset of $X$, see~\cite{ELMNP}. 

\begin{lem}\label{lem:FSenv} Suppose $\f\colon X^\an\to\R$ is bounded, with lsc regularization $\f_\star\colon X^\an\to\R$. Then $\qq_L(\f)=\qq_L(\f_\star)\le\env_L(\f_\star)$, and equality holds outside $\B_+(L)$. 
\end{lem}
In particular, $\qq_L(\f)=\env_L(\f_\star)$ when $L$ is ample. In this case, $\qq_L$ coincides with the envelope $\qq_{c_1(L)}$ in~\cite[\S5.3]{trivval}.
\begin{proof} Since any function $\p\in\cH^\gf(L)$ is continuous, it satisfies $\p\le\f$ iff $\p\le\f_\star$. Thus $\qq_L(\f)=\qq_L(\f_\star)$, and we may therefore assume wlog that $\f$ is lsc. Since $\cH^\gf(L)\subset\PSH(L)$, we trivially have $\qq_L(\f)\le\env_L(\f)$. Conversely, pick $\p\in\PSH(L)$ such that $\p\le\f$. Let $E$ be an effective $\Q$-Cartier divisor such that $A:=L-E$ is ample. By~\cite[Theorem~4.15]{trivval}, we can write $\p$ as the pointwise limit of a decreasing net $(\p_j)$ in $\cH^\gf(L+\e_j A)$ with $\e_j\to 0$. Pick $\e>0$, so that $\p<\f+\e$. As in the proof of~Lemma~\ref{lem:envlsc}, since $\p_j$ is usc and $\f$ is lsc, a simple variant of Dini's lemma shows that $\p_j<\f+\e$ for all $j$ large enough.
  
  Set $\log|s_E|:=m^{-1}\log|s_{mE}|$, where $s_{mE}$ is the canonical global section of $\cO_X(mE)$ for any $m\ge1$ such that $mE$ is integral.
Then $\log|s_E|\le 0$ lies in $\cH^\gf(E)$, so it follows that $\tau_j:=(1+\e_j)^{-1}(\p_j+\e_j\log|s_E|)$ lies in $\cH^\gf(L)$. Further, 
$$
\tau_j\le(1+\e_j)^{-1}(\f+\e)\le\f+\e+C\e_j
$$
for some uniform $C>0$, since $\f$ is bounded, and hence 
$$
\tau_j\le\qq_L(\f+\e+C\e_j)=\qq_L(\f)+\e+C\e_j.
$$
We have thus proved
$\p_j+\e_j\log|s_E|\le(1+\e_j)(\qq_L(\f)+\e+C\e_j)$; at any point of 
$$
(X-E)^\an=\{\log|s_E|>-\infty\},
$$
this yields $\p\le\qq_L(\f)$, and hence $\env_L(\f)\le\qq_L(\f)$, which proves the result. 
\end{proof}
%
%%%%%%%%%%%%%%%%%%%%%%%%%%%%%%%%%%%%%%%%%%%%%%%%%%%%%%%%%%%%%%%%%%%
%
\subsection{Envelopes from test configurations}
Let $L$ be a big line bundle. Any test configuration $(\cX,\cL)$ for $(X,L)$  defines a function $\f_\cL\in\PL$, and we seek to compute the Fubini--Study envelope $\qq_L(\f_\cL)$.

To this end, we introduce a slight generalization of the definitions in~\cite[\S2.1]{trivval}. To any $\Gm$-invariant ideal $\fa\subset\cO_\cX$, we attach a function $\f_\fa\colon X^\an\to [-\infty,0]$ by setting $\f_{\fa}(v):=-\sigma(v)(\fa)$, where $\sigma=\sigma_\cX$ denotes Gauss extension (see~\cite[Remark~1.9]{trivval}). In terms of the weight decomposition $\fa=\sum_{\la\in\Z_{\ge 0}}\fa_\la\unipar^{-\la}$ with $\fa_\la\subset\cO_X$, we have $\f_\fa=\max_\la\{\log|\fa_\la|+\la\}$. If $\cL$ is an honest line bundle such that $\cL\otimes\fa$ is globally generated, one easily checks as in~\cite[Proposition~2.25]{trivval} that $\f_\cL+\f_\fa$ lies in $\cH^\gf_\Q(L)$.

\begin{lem}\label{lem:envline} Let $L$ be a big line bundle on $X$, and $(\cX,\cL)$ an integrally closed test configuration for $(X,L)$. For each sufficiently divisible $m\in\Z_{>0}$, denote by $\fa_m\subset\cO_\cX$ the base ideal of $m\cL$, and set $\f_m:=\f_\cL+m^{-1}\f_{\fa_m}$. Then $\f_m\in\cH^\gf_\Q(L)$ and $(\f_m)_m$ forms an increasing net of functions on $X^\an$ converging pointwise to $\qq_L(\f_\cL)$.
\end{lem}
Here we consider $(\f_m)_m$ as a net indexed by the set $m_0\Z_{>0}$ for some sufficiently divisible $m_0$, and partially ordered by divisibility.

To prove the lemma, recall~\cite[\S1.2]{trivval}  that if $\cL$ (and hence $L$) is an honest line bundle, then $\Hnot(\cX,\cL)$ lies as a $k[\unipar]$-submodule of $\Hnot(X,L)_{k[\unipar^{\pm1}]}$. The next result provides a valuative characterization of this submodule in terms of~$\f_\cL$. 

\begin{lem}\label{lem:fcL} Assume $\cL$ is an honest line bundle, pick $s\in\Hnot(X,L)_{k[\unipar^{-\pm1}]}$, and write $s=\sum_{\la\in\Z} s_\la\unipar^{-\la}$ with $s_\la\in\Hnot(X,L)$. Then $s\in\Hnot(\cX,\cL)$ iff $\max_\la\{\log|s_\la|+\la\}\le\f_\cL$ on $X^\an$. 
\end{lem}
\begin{proof} By $\Gm$-invariance, we have $s\in\Hnot(\cX,\cL)\Leftrightarrow s_\la\unipar^{-\la}\in\Hnot(\cX,\cL)$ for all $\la\in\Z$, and we may thus assume $s=s_\la\unipar^{-\la}$ for some $\la\in\Z$. 

Since $\cX$ is integrally closed, we have $\rho_\star\cO_{\cX'}=\cO_{\cX}$, and hence $\Hnot(\cX',\rho^\star\cL)=\Hnot(\cX,\cL)$, for any higher test configuration $\rho\colon\cX'\to\cX$ (see the proof of~\cite[Proposition~2.30]{trivval}). After pulling back $\cL$ to a higher test configuration, we may thus assume that $\cX$ dominates the trivial test configuration via $\mu\colon\cX\to\cX_\triv$. Set $D:= \cL-\mu^\star\cL_\triv$, so that $\f_\cL=\f_D$. Viewed as a rational section of $\cL$, $s$ is regular outside $\cX_0$. For any $v\in X^\an$ with Gauss extension $w=\sigma(v)$, we further have 
$$
w(s)=v(s_\la)-\la+w(D)=-\log|s_\la|(v)-\la+\f_D(v).
$$
If $s$ is a regular section, then $w(s)\ge 0$, and hence $\log|s_\la|(v)+\la\le\f_D(v)$ for any $v\in X^\an$. Conversely, the latter condition implies $b_E^{-1}\ord_E(s)=-\log|s_\la|(v_E)-\la+\f_D(v_E)\ge 0$ for each irreducible component $E$ of $\cX_0$, since $\sigma(v_E)=b_E^{-1}\ord_E$; this yields, as desired, $s\in\Hnot(\cX,\cL)$ (compare~\cite[Lemma~1.23]{trivval}).
\end{proof}

\begin{proof}[Proof of Lemma~\ref{lem:envline}] Replacing $L$ and $\cL$ by sufficiently divisible multiples, we may assume that $L$ and $\cL$ are honest line bundles.

  We have $\fa_m\cdot\fa_{m'}\subset\fa_{m+m'}$ for all $m,m'\in\N$. This implies that the net $(\f_m)_m$ is increasing.

   By definition of $\fa_m$, $m\cL\otimes\fa_m$ is globally generated. As noted above, this implies $\f_{m\cL}+\f_{\fa_m}\in\cH^\gf_\Q(mL)$, and hence $\f_m\in\cH^\gf_\Q(L)$. Since $\f_{\fa_m}\le 0$, we further have $\f_m\le\f_\cL$, and hence $\f_m\le\qq_L(\f_\cL)$, see~\eqref{equ:FSenv}.
 
Conversely, pick $\p\in\cH^\gf_\Q(L)$ such that $\p\le\f_\cL$, and write $\p=\frac 1m\max_i\{\log|s_i|+\la_i\}$ for a finite set of nonzero sections $s_i\in\Hnot(X,mL)$ and $\la_i\in\Z$. For each $i$, we then have $\log|s_i|+\la_i\le m\f_{\cL}=\f_{m\cL}$, and hence $s_i\unipar^{-\la_i}\in\Hnot(\cX,m\cL)$, see Lemma~\ref{lem:fcL}. Since $\fa_m$ is locally generated by $\Hnot(\cX,m\cL)$, this implies in turn $\log|s_i|+\la_i\le \f_{m\cL}+\f_{\fa_m}$, and hence $\p\le\f_m$. Taking the supremum over $\p$, we conclude, as desired, $\qq_L(\f_\cL)\le\sup_m\f_m$. 
\end{proof}
%
%%%%%%%%%%%%%%%%%%%%%%%%%%%%%%%%%%%%%%%%%%%%%%%%%%%%%%%%%%%%%%%%%%%
%
\subsection{The energy pairing}\label{sec:enpair}
Various incarnations of the energy pairing play a key role in~\cite{trivval}. First of all, when $\theta_0,\dots,\theta_n\in\Num(X)$ are arbitrary numerical classes and $\f_0,\dots,\f_n\in\PL(X)_\R$ are ($\R$-linear combinations of) PL functions, then
\begin{equation*}
  (\theta_0,\f_0)\cdot\ldots\cdot(\theta_n,\f_n)\in\R
\end{equation*}
is defined as an intersection number on a compactified test configuration for $X$, see~\cite[\S3.2]{trivval}. 
The following result would naturally belong to~\cite[Proposition~3.14]{trivval}.
\begin{lem}\label{lem:PL}
  Let $\pi\colon Y\to X$ be a projective birational morphism, $\theta_0,\dots,\theta_n\in\Num(X)$ numerical classes, and $\f_0,\dots,\f_n\in\PL(X)$ PL functions. Then
  \begin{equation*}
    (\theta_0,\f_0)\cdot\ldots\cdot(\theta_n,\f_n)
    =(\pi^\star\theta_0,\pi^\star\f_0)\cdot\ldots\cdot(\pi^\star\theta_n,\pi^\star\f_n).
  \end{equation*}
\end{lem}
\begin{rmk}
  While we are assuming that $X$ and $Y$ are irreducible, the result holds even without this assumption, as in~\cite[Proposition~3.14]{trivval}.
\end{rmk}
\begin{proof}
  There exists a test configuration $\cX$ for $X$ that dominates $\cX_\triv=X\times\A^1$,  and vertical $\Q$-Cartier divisor $D_i\in\VCar(\cX)_\Q$ that determine the functions $\f_i$, 
  $0\le i\le n$. Then
  \[
    (\theta_0,\f_0)\cdot\ldots\cdot(\theta_n,\f_n)
    =(\theta_{0,\bar\cX}+D_0)\cdot\ldots\cdot(\theta_{n,\bar\cX}+D_n),
  \]
  where the intersection number is computed on the canonical compactification $\bar\cX\to\P^1$ and $\theta_{i,\cX}\in\Num(\bar\cX)$ denotes the pullback of $\theta_i$. The canonical birational map $\cY_\triv=Y\times\A^1\dashrightarrow\cX$ being $\Gm$-equivariant, we can choose a test configuration $\cY$ for $Y$ that dominates $\cY_\triv$ such that $\pi\colon Y\to X$ extends to a $\Gm$-equivariant morphism $\pi\colon \cY\to\cX$. Then $\pi^\star\f_{D_i}=\f_{\pi^\star D_i}$ for all $i$, and we have
  \begin{multline*}
    (\pi^\star\theta_0,\pi^\star\f_0)\cdot\ldots\cdot(\pi^\star\theta_n,\pi^\star\f_n)
    =(\pi^\star\theta_{0,\bar\cX}+\pi^\star D_0)\cdot\ldots\cdot(\pi^\star\theta_{n,\bar\cX}+\pi^\star D_n)\\
    =(\theta_{0,\bar\cX}+D_0)\cdot\ldots\cdot(\theta_{n,\bar\cX}+D_n)
    =(\theta_0,\f_0)\cdot\ldots\cdot(\theta_n,\f_n),
\end{multline*}
where the second equality follows from the projection formula.
\end{proof}
In~\cite[\S7]{trivval}, the energy pairing was extended in various ways. First, one can define
\begin{equation*}
  (\om_0,\f_0)\cdot\ldots\cdot(\om_n,\f_n)\in\R\cup\{-\infty\}
\end{equation*}
for $\om_i\in\Amp(X)$ and $\f_i\in\PSH(\om_i)$ by approximation from above by functions in $\PSH(\om_i)\cap\PL(X)$. Given $\om\in\Amp(X)$, a function $\f\in\PSH(\om)$ has \emph{finite energy} if $(\om,\f)^{n+1}>-\infty$, and the set of such functions is denoted by $\cE^1(\om)$. If $\f\in\PSH(\om)$, we set
\begin{equation*}
  \en_\om(\f):=\frac{(\om,\f)^{n+1}}{(n+1)(\om^n)}.
\end{equation*}
The functional $\en_\om$ is increasing and satisfies $\en_\om(\f+c)=\en_\om(\f)+c$ for any $\f\in\PSH(\om)$ and $c\in\R$. 
We have $(\om_0,\f_0)\cdot\ldots\cdot(\om_n,\f_n)>-\infty$ for any $\om_i\in\Amp(X)$ and $\f_i\in\cE^1(\om_i)$. 

For a general bounded-above function $\f\colon X^\an\to\R\cup\{-\infty\}$ we set
\[
  \en_\om(\f):=\sup\{\en_\om(\p)\mid\p\in\PSH(\om), \p\le\f\}.
\]
Then $\en_\om(\f)=\en_\om(\env_\om(\f))$ for any bounded-above function $\f$.

A function $\f\colon X^\lin\to\R$ is said to be of finite energy if it is of the form $\f=\f^+-\f^-$, where $\f^\pm\in\cE^1(\om)$ for some $\om\in\Amp(X)$. The energy pairing then extends as a (finite) multilinear pairing $(\theta_0,\f_0)\cdot\ldots\cdot(\theta_n,\f_n)$ for arbitrary numerical classes $\theta_i\in\Num(X)$ and functions $\f_i$ of finite energy.
% 
%
%%%%%%%%%%%%%%%%%%%%%%%%%%%%%%%%%%%%%%%%%%%%%%%%%%%%%%%%%%%%%%%%%%%
%
%
\section{Theorem~A}
We now prove Theorem~A and derive some consequences.
%
%%%%%%%%%%%%%%%%%%%%%%%%%%%%%%%%%%%%%%%%%%%%%%%%%%%%%%%%%%%%%%%%%%%
%
\subsection{Proof of Theorem~A}
  The result is trivial if $\theta$ is not pseudoeffective, as $\PSH(\theta)$ is then empty. Otherwise, we can write $\theta=\lim_i c_1(L_i)$ for a sequence of big $\Q$-line bundles $L_i$ with $c_1(L_i)\ge\theta$; by~\cite[Lemma~5.9]{trivval}, we may thus assume that $\theta=c_1(L)$ for a big $\Q$-line bundle $L$. Pick $\f\in\PL(X)$. By~Lemma~1.3, we need to show that $\env_L(\f)$ is $L$-psh. By~\cite[Theorem~2.31]{trivval}, we have $\f=\f_\cL$ for some integrally closed test configuration $(\cX,\cL)$ for $(X,L)$. After replacing $L$ with a multiple, we may further assume that $L$ and $\cL$ are honest line bundles.  

Since we assume that $\charac k=0$ or $\dim X\le 2$ (and hence $\dim\cX\le 3$), we can rely on resolution of singularities and assume that $\cX$ is smooth and $\cX_0$ has simple normal crossings support. Assume first that $\charac k=0$, and let $\fb_m$ be the multiplier ideal of the graded sequence $\fa_\bullet^m$, see Lemma~\ref{lem:envline}. The inclusion $\fa_m\subset\fb_m$ is elementary, 
  and we have $\fb_{ml}\subset\fb_m^l$ for all $m,l$ by the subadditivity property of 
  multiplier ideals. This implies that 
  \begin{equation*}
 (ml)^{-1}\f_{\fa_{ml}}\le (ml)^{-1}\f_{\fb_{ml}}\le m^{-1}\f_{\fb_m}
  \end{equation*}
  for all $m$ and $l$. Letting $l\to\infty$ shows that
\begin{equation}\label{equ:squeeze}
  \f_m\le\qq_L(\f_\cL)\le\p_m:=\f_\cL+m^{-1}\f_{\fb_m}
  \end{equation}    
  for all  $m$, by Lemma~\ref{lem:envline}. By the uniform global generation property of multiplier ideals, we can find a $\Gm$-equivariant ample line bundle $\cA$ on $\cX$ such that $\cO_\cX(m\cL+\cA)\otimes\fb_m$ is  globally generated for all $m$. As noted before Lemma~\ref{lem:envline}, this implies $\f_{m\cL+\cA}+\f_{\fb_m}\in\cH^\gf(mL+A)$, with $A\in\Pic(X)$ the restriction of $\cA$, and hence 
$$
\p'_m:=\p_m+\tfrac 1m\f_\cA\in\cH^\gf_\Q(L+\tfrac 1m A).
$$
After adding to $\cA$ a multiple of $\cX_0$, we may further assume $\f_\cA\ge 0$, which guarantees that the net $(\p'_m)$ is decreasing with respect to the divisibility order, and hence that $\p:=\inf_m \p'_m$ is either $L$-psh or identically $-\infty$ (see~\cite[Theorem~4.5]{trivval}). By~\eqref{equ:squeeze}, we have 
$$
\qq_L(\f_\cL)\le\p'_m\le\f_\cL+\tfrac 1m\f_\cA,
$$
and hence $\qq_L(\f_\cL)\le\p\le\f_\cL$. In particular, $\p\not\equiv-\infty$, so $\p\in\PSH(L)$, and hence $\p\le\env_L(\f_\cL)$. Finally, pick $\tau\in\PSH(L)$ such that $\tau\le\f_\cL$. By Lemma~\ref{lem:FSenv}, we have $\tau\le\env_L(\f_\cL)=\qq_L(\f_\cL)\le\p$ on a Zariski open subset of $X^\an$, and hence on $X^\div$. Since $\tau$ and $\p$ are $L$-psh, it follows from~\cite[Theorem~4.22]{trivval} that $\tau\le\p$ on $X^\an$. Taking the sup over $\tau$ yields $\env_L(\f_\cL)\le\p$, and we conclude, as desired, that $\env_L(\f_\cL)=\p$ is $L$-psh. 

When $\charac k>0$, the very same argument applies with test ideals in place of multiplier ideals, see~\cite{GJKM19} for details. 
%
%%%%%%%%%%%%%%%%%%%%%%%%%%%%%%%%%%%%%%%%%%%%%%%%%%%%%%%%%%%%%%%%%%%
%
\subsection{Consequences}
We now list some consequences of Theorem~A. First, we can characterize psef classes, similarly to the complex analytic case.
\begin{cor}
  Assume that $X$ satisfies the assumptions in Theorem~A. Then, for any $\theta\in\Num(X)$,  we have $\PSH(\theta)\ne\emptyset$ iff $\theta$ is psef. Moreover, in this case, the function $$V_\theta:=\pp_\theta(0)$$ is $\theta$-psh.
\end{cor}
\begin{proof}
  It follows from~\cite[Definition~4.1]{trivval} that $\PSH(\theta)\ne\emptyset$ only if $\theta$ is psef. First suppose $\theta$ is big. By Theorem~A, $V_\theta:=\pp_\theta(0)$ is $\theta$-psh. Note that $V_\theta(v_\triv)=\sup V_\theta=0$, where $v_\triv$ is the trivial valuation on $X$.

  Now suppose $\theta$ is merely psef, and pick a sequence $(\theta_m)_1^\infty$ of big classes converging to $\theta$, such that $\theta\le\theta_{m+1}\le\theta_m$ for all $m$. As $\PSH(\theta_{m+1})\subset\PSH(\theta_m)$ for all $m$, the sequence $(V_{\theta_m})_m$ is pointwise decreasing on $X^\an$. Let $\f$ be its limit. We have $\sup\f=\f(v_\triv)=0$, and $\f\in\PSH(\theta_m)$ for every $m$. It now follows from~\cite[Theorem~4.5]{trivval} that $\f\in\PSH(\theta)$. Finally, it is easy to see that $\f=\pp_\theta(0)$. Indeed, $\f\le0$, and if $\p\in\PSH(\theta)$ satisfies $\p\le0$, then $\p\in\PSH(\theta_m)$ for all $m$, so $\p\le V_{\theta_m}$, and hence $\p\le\f$.
\end{proof}  

By~\cite[Theorem~5.11]{trivval}, Theorem~A now implies the following compactness result.
\begin{cor}\label{cor:contenvcomp} Under the assumptions on $X$ of Theorem~A, the set
$$
\PSH_{\sup}(\theta)=\left\{\f\in\PSH(\theta)\mid\sup\f=0\right\}
$$
is compact for any psef class $\theta\in\Num(X)$.
\end{cor}

Finally, as an immediate consequence of Theorem~A and~\cite[Theorem~6.31]{trivval}, we have the following version of Siu's decomposition theorem.
\begin{cor}\label{cor:Siu} Suppose that $X$ satisfies the assumptions of Theorem~A. Pick $\theta\in\Num(X)$ and an effective $\Q$-Cartier divisor $E$. Then, for any $\f\in\PSH(\theta)$, we have: 
$$
\f\le\log|s_E|+O(1)\Longleftrightarrow\f-\log|s_E|\in\PSH(\theta-E).
$$ 
\end{cor}
As before, $\log|s_E|=m^{-1}\log|s_{mE}|$, where $s_{mE}$ is the canonical global section of $\cO_X(mE)$ for any $m\ge1$ such that $mE$ is integral.

%
% 
%%%%%%%%%%%%%%%%%%%%%%%%%%%%%%%%%%%%%%%%%%%%%%%%%%%%%%%%%%%%
%
%
\section{Proof of Theorem~B}
We start by proving:
\begin{lem}\label{lem:enpairpull} Let $\pi\colon\tilde X\to X$ be a projective birational morphism, and pick a bounded $\om$-psh function $\p$. Then $(\om,\p)^{n+1}=(\pi^\star\om,\pi^\star\p)^{n+1}$. 
\end{lem}
Here $\pi^\star\om$ may not be ample, but the right hand side is well-defined, as $\pi^\star\p$ is a function of finite energy. In fact $\pi^\star\p\in\cE^1(\tilde\om)$ for any ample class $\tilde\om\ge\pi^\star\om$.
\begin{proof}
  The case when $\p\in\PL(X)$ follows from Lemma~\ref{lem:PL}.
  In the general case, write $\p$ as the pointwise limit of a decreasing net $(\p_j)$ in $\PL\cap\PSH(\om)$, and pick $\tilde\om\in\Amp(\tilde X)$ such that $\tilde\om\ge\pi^\star\om$. Then $\pi^\star\p_j$ decreases to $\pi^\star\p$ pointwise on $\tilde{X}^\an$. Moreover, $\pi^\star\p_j$ and $\pi^\star\p$ are $\tilde\om$-psh, and hence lie in $\cE^1(\tilde\om)$ as they are bounded.  By~\cite[Theorem~7.14~(iii)]{trivval} we have 
$(\om,\p_j)^{n+1}\to(\om,\p)^{n+1}$ and $(\pi^\star\om,\pi^\star\p_j)^{n+1}\to(\pi^\star\om,\pi^\star\p)^{n+1}$.
Now $(\pi^\star\om,\pi^\star\p_j)^{n+1}=(\om,\p_j)^{n+1}$ for all $j$ by the PL case, and the result follows.
\end{proof}
As stated in the introduction, we introduce:
\begin{defi}
  Let $X$ be a projective variety, and $\om\in\Num(X)$ an ample class. We say that $(X,\om)$ has the weak envelope property if there exists a projective birational morphism $\pi\colon\tilde X\to X$, and an ample class $\tilde\om\in\Num(\tilde X)$, such that $\tilde\om\ge\pi^\star\om$ and $(\tilde X,\tilde\om)$ has the envelope property.
\end{defi}
\begin{lem}
  If $\charac k=0$ or $\dim X\le 2$, then any ample class $\om\in\Num(X)$ has the weak envelope property.
\end{lem}
\begin{proof}
  In both cases, we can pick $\pi\colon\tilde X\to X$ as a resolution of singularities, and then pick any ample class $\tilde\om\ge\pi^\star\om$. By~\cite[Theorem~5.20]{trivval} (or Theorem~A), the envelope property holds for $(\tilde X,\tilde\om)$, and we are done.
\end{proof}
\begin{proof}[Proof of Theorem~B] Set $\tau:=\env_\om(\f)$. For any $\p\in\PSH(\om)$, we have $\p\le\f\Longleftrightarrow\p\le\tau$, and hence $\en_\om(\f)=\en_\om(\tau)\le\en_\om(\tau^\star)$. Since $\tau$ is the pointwise supremum of the family $\cF=\{\p\in\PSH(\om)\mid\p\le\f\}$, and since $\cF$ is stable under finite max, we can find an increasing net $(\p_i)$ of $\om$-psh functions such that $\sup_i\p_i=\tau$ pointwise on $X^\an$. Replacing $\p_i$ with $\max\{\p_i,\inf\p\}$, we can further assume that $\p_i$ is bounded. 

  By assumption, we can find a projective birational morphism $\pi\colon\tilde X\to X$, and an ample class $\tilde\om\in\Num(\tilde X)$ such that $\tilde\om\ge\pi^\star\om$ and $(\tilde X,\tilde\om)$ has the envelope property.
  Now $\tilde\tau:=\pi^\star\tau=\sup_i\pi^\star\p_i$ with $\pi^\star\p_i\in\PSH(\tilde\om)$, and it follows that $\tilde\tau^\star$ is $\tilde\om$-psh, and coincides with $\tilde\tau=\sup_i\pi^\star\p_i=\lim_i\sup\pi^\star\p_i$ on $\tilde X^\div$. By~\cite[Theorem~7.38]{trivval}, we get $(\pi^\star\om,\pi^\star\p_i)^{n+1}\to(\pi^\star\om,\tilde\tau^\star)^{n+1}$. On the other hand, Lemma~\ref{lem:enpairpull} yields 
\begin{equation*}
(\pi^\star\om,\pi^\star\p_i)^{n+1}=(\om,\p_i)^{n+1}=(n+1)\vol(\om)\en_\om(\p_i)\le (n+1)\vol(\om)\en_\om(\tau),
\end{equation*}
and we infer
\begin{equation}\label{equ:entau}
(\pi^\star\om,\tilde\tau^\star)^{n+1}\le (n+1)\vol(\om)\en_\om(\tau). 
\end{equation}
By~\cite[Theorem~5.6]{trivval} we also have $\tau^\star=\tau$ on $X^\div$. Each  $\p\in\PSH(\om)$ such that $\p\le\tau^\star$ on $X^\an$ therefore satisfies $\p\le\tau$ on $X^\div$ (see~\cite[Theorem~5.6]{trivval}); hence $\pi^\star\p\le\tilde\tau\le\tilde\tau^\star$ on $\tilde X^\div$, which implies $\pi^\star\p\le\tilde\tau^\star$ on $\tilde X^\an$ (see~\cite[Theorem~4.22]{trivval}). Assuming $\p$ bounded, we get
$$
(\om,\p)^{n+1}=(\pi^\star\om,\pi^\star\p)^{n+1}\le(\pi^\star\om,\tilde\tau^\star)^{n+1},
$$
where the equality follows from Lemma~\ref{lem:enpairpull}, and the inequality from the monotonicity of the energy pairing, see~\cite[Theorem~7.1]{trivval}.
Taking the supremum over $\p$ now yields 
$$
(n+1)\vol(\om) \en_\om(\tau^\star)\le(\pi^\star\om,\tilde\tau^\star)^{n+1}. 
$$
Combined with~\eqref{equ:entau}, this implies $\en_\om(\tau^\star)\le\en_\om(\tau)$, and the result follows. 
\end{proof}

%
% 
%%%%%%%%%%%%%%%%%%%%%%%%%%%%%%%%%%%%%%%%%%%%
%
% 

%
% 
%
%%%%%%%%%%%%%%%%%%%%%%%%%%%%%%%%%%%%%%%%%%%%%%%%%%%%%%%%%%%%%%%%%%%
%
%
%
\end{document}